\documentclass{amsart}

\usepackage[english]{babel}

\usepackage{amsmath, amssymb, amsthm, enumitem, comment}
\usepackage{graphicx}
\usepackage{tikz-cd}
\usepackage[all]{xy}
\usepackage[colorlinks=true, allcolors=blue]{hyperref}

\newtheorem*{claim-star}{Claim}
\newtheorem{theorem}{Theorem}[section] 
\newtheorem{lemma}[theorem]{Lemma}

\newtheorem{prop-def}[theorem]{Proposition-Definition}

\newtheorem{fact-eh}[theorem]{Fact(?)}

\newtheorem{proposition}[theorem]{Proposition}
\newtheorem{proposition-eh}[theorem]{Proposition(?)}
\newtheorem*{theorem-star}{Theorem}
\newtheorem*{conjecture-star}{Conjecture}
\newtheorem*{question-star}{Question}
\newtheorem*{lemma-star}{Lemma}

\theoremstyle{definition}
\newtheorem{definition}[theorem]{Definition}
\newtheorem{construction}[theorem]{Construction}

\newtheorem{problem}[theorem]{Problem}

\newtheorem{question}[theorem]{Question}

\numberwithin{equation}{section}
\newtheorem{remark}[theorem]{Remark}
\theoremstyle{remark}

\newcommand{\Ff}{\mathbb{F}}

\newcommand{\Zz}{\mathbb{Z}}

\newcommand{\Ee}{\mathbb{E}}
\newcommand{\Ss}{\mathbb{S}}

\newcommand{\calC}{\mathcal{C}}

\newcommand{\Rom}[1]{\uppercase\expandafter{\romannumeral #1\relax}}

\newcommand{\calA}{\mathcal{A}}

\newcommand{\suspen}{\Sigma^{\infty}}

\newcommand{\CAlg}{\operatorname{CAlg}}
\newcommand{\Mod}{\operatorname{Mod}}

\DeclareMathOperator{\Fun}{Fun}

\DeclareMathOperator{\colim}{colim}
\DeclareMathOperator{\Sp}{Sp}
\DeclareMathOperator{\MU}{MU}
\DeclareMathOperator{\BP}{BP}
\DeclareMathOperator{\HH}{HH}
\DeclareMathOperator{\K}{K}
\DeclareMathOperator{\THH}{THH}

\DeclareMathOperator{\TC}{TC}

\title{A note on higher topological Hochschild homology}
\author{Rixin Fang}
\email{rxfang21@m.fudan.edu.cn}
\address{Shanghai Center for Mathematical Sciences, Fudan University, 2005 Songhu Road 200438, Shanghai, China}

\begin{document}
\begin{abstract}
Chromatic redshift phenomena suggest that algebraic K-theory increases the height of a commutative ring spectrum by one. In many cases, the chromatic redshift is already detected by negative topological cyclic homology. This paper explores higher chromatic redshift via homotopy fixed points spectrum of higher topological Hochschild homology. Specifically, starting from a commutative ring spectrum that detects $v_n$-elements, the homotopy fixed points spectrum of higher topological Hochschild homology of it detects $v_{n+k}$-elements, with $k$ greater than one.
\end{abstract}
\maketitle

\section{Introduction}
\subsection{Motivation}
Chromatic redshift states that algebraic K-theory increases the height by one. The qualitative form of chromatic redshift is resolved for $\Ee_{\infty}$-rings. 
For a commutative ring spectrum $X$, J.Hahn proved that if $K(n)_*X =0$ then $K(n+1)_* X=0$.

Hence, one can make the following definition.

\begin{definition}
    A commutative ring spectrum $X$ is height $n$ if $K(n)_* X \ne 0$ and $K(n+1)_* X=0$. We write the height of $X$ by $\mathrm{ht}(X)$.
\end{definition}

From this definition, it is natural to ask whether there exists higher redshift phenomena. Meaning that, is there a functor 
\begin{equation*}
  F^n:  \CAlg \to \CAlg
\end{equation*}
such that $\mathrm{ht}(F^n(X))=\mathrm{ht}(X)+n$?

It is clear that $\K^{(n)}$ has this property.

The redshift phenomenon is observed in many known cases through negative topological cyclic homology.
For instance, we have the following theorem.

\begin{theorem}[{\cite{HW22}}]
Let $\BP\langle n\rangle$ be the $\Ee_3$-algebra forms constructed in \cite{HW22}. Then \begin{equation*}
    K(n+1)\otimes \THH(\BP\langle n\rangle)^{h S^1} \ne 0.
\end{equation*}
\end{theorem}

\begin{remark}
In \cite{HW22}, the authors in fact proved 
\begin{equation*}
    K(n+1)\otimes \THH(\BP\langle n\rangle/\MU)^{h S^1} \ne 0.
\end{equation*}
Because there is a ring map
\begin{equation*}
    \THH(\BP\langle n\rangle)^{h S^1}\to \THH(\BP\langle n\rangle/\MU)^{h S^1},
\end{equation*}
hence this implies the above stated theorem.
For $n=1$, there is a weaker form proved by Ausoni--Rognes \cite{AR02}.
They proved the map
\begin{equation*}
    V(1)_* \to V(1)_* \THH(\ell)^{h S^1}
\end{equation*}
sends $v_2$ to a non-zero element.

It seems that it is possible to using their argument to prove that the map
\begin{equation*}
    k(2)_* \to k(2)_* \THH(\ell)^{h S^1}
\end{equation*}
sends $v_2$ to a non-zero element.
\end{remark}
There is a detection theorem for height one Lubin--Tate theory (connective $p$-complete complex K-theory).
\begin{theorem}[{\cite{Ausoniku}}]
The map \begin{equation*}
    V(1)_* \to V(1)_* \THH(ku_p)^{h S^1}
\end{equation*}
sends $v_2$ to a non-zero element.
\end{theorem}
Recall that there is a trace map
\begin{equation*}
    \K^{(n)} \to \THH^{(n)}(A)^{ h T^n}
\end{equation*}
which is a ring map. Here $\THH^{(n)}(A)=\THH(\THH^{(n-1)}(A))$, and $\THH^{(1)}(A)=\THH(A)$. The target $\THH^{(n)}(A)$ can be identified with the Loday construction $T^n\otimes A= L_{T^n} A$. This is referred to higher topological Hochcshild homology. In the height $-1$ case, the computation of higher topological  Hochschild homology was done by Veen. 
For the height $0$ case, see \cite{DLR18}.
Veen proved the following theorem.
\begin{theorem}[{\cite{Veen18}}]
Let $p\geq 5$ and $1\leq n\leq p$ or $p=3$ and $1\leq n\leq 2$.
The unit map \begin{equation*}
    k(n-1)_*\to k(n-1)_*(L_{T^n} \Ff_p)^{h T^n}
\end{equation*}
sends $v_{n-1}$ to a nonzero element in the right hand side.
\end{theorem}

Therefore, it is natural to ask same question replacing $\Ff_p$ by $\BP\langle n\rangle$ (when there is an $\Ee_{\infty}$-form of $\BP\langle n\rangle$). 

This paper is devoted to answer the following question.
\begin{question}\label{question:detection}
Let $X$ be a height $n$ commutative ring spectrum. Is 
\begin{equation*}
    k(n+m)_* (L_{T^m} X)^{h T^m}
\end{equation*}
non-zero?
\end{question}

\subsection{Outline of the paper}
In Section \ref{section2}, we define the Loday construction for commutative algebras in a presentable stable symmetric monoidal category. In Section \ref{section3}, using several ring maps and the result of Veen, we prove the following result.
\begin{theorem}[\autoref{thm:higherredshift}]
Let $p\geq 5$ and $1\leq n+2\leq p$. Then the unit map
\begin{equation*}
k(n+1)_* \to k(n+1)_* (L_{T^n} j_p)^{h T^n}
\end{equation*}
sends $v_{n+1}$ to a nonzero element.
\end{theorem}

In Section \ref{section4}, we introduce some questions may related to the Question \ref{question:detection}.

\subsection{Acknowledgment}
The author would like to express his sincere gratitude to his advisor, Guozhen Wang, for his continuous support. I am grateful to Gabriel Knoll for helpful comments on an early draft and for sharing the idea on Definition \ref{def:Loday}. Special thanks go to John Rognes for valuable communications and for suggesting the key step (Remark \ref{trick}) in the proof of the main result. Finally, I would like to thank Jingbang Guo and Foling Zou for several useful conversations.
\section{Loday construction and higher topological Hochschild homology}\label{section2}
Given a commutative ring spectrum $A$, and a finite simplicial set $X$. We can define the Loday construction $X\otimes A \in \CAlg$ as follows. For a finite set $Y$, we can define $Y\otimes A= \otimes_{y\in Y} A$. Then we define \begin{equation*}
    X\otimes A:= \colim_{[n]\in \triangle} X([n])\otimes A \in \CAlg.
\end{equation*}
Since any simplicial set $Y$ can be written as a colimit of finite simplicial sets $Y=\colim Y_{\alpha}$, then one defines \begin{equation*}
    Y\otimes A: = \colim Y_{\alpha}\otimes A.
\end{equation*}
This construction can be more explicitly.
\begin{construction}
Given a symmetric monoidal $\infty$-categroy $\calC^{\otimes}$.
We can consider the following $\infty$-operad. 
$\calC_{act}^{\otimes}:= \calC^{\otimes}\times_{\mathrm{Fin}_*} \mathrm{Fin}$.

Recall that a commutative algebra object $R$ is a section of $\calC^{\otimes} \to \mathrm{Fin}_*$ satisfying certain  conditions, i.e. $R\in \Fun_{\mathrm{lax}}(\mathrm{Fin}_*, \calC^{\otimes})$.
This gives a functor $R\in \Fun(\mathrm{Fin},\calC_{act}^{^{\otimes}})$.

By definition, this gives $R\in \Fun(\mathrm{Fin}, \calC^{\otimes}_{act})$. We now define $R^{\otimes}$ to be 
\begin{equation}\label{equ:tensor}
    R^{\otimes} : \mathrm{Fin} \overset{R}{\to} \calC_{act}^{\otimes} \overset{\otimes}{\to} \calC.
\end{equation}

The last functor is given by the image of $\mathrm{id}_{\calC}$ under the equivalence \cite[Proposition III.3.2]{NS18}
\begin{equation*}
    \Fun_{\otimes} (\calC^{\otimes}_{act},\calC) \simeq \Fun_{\mathrm{lax}}(\calC,\calC).
\end{equation*}

Now, for a presentable symmetric monoidal $\infty$-category $\calC$, given a finite simplicial set $X\in \Fun(\triangle^{op}, \mathrm{Fin})$. We can form the following construction
\begin{equation*}
    R^{\otimes X}:= \colim( \triangle^{op} \overset{X}{\to}  \mathrm{Fin} \overset{R}{\to} \calC^{\otimes}_{act} \overset{\otimes}{\to} \calC).
\end{equation*}
This is essentially the construction in \cite{coveringhomology}.
Note that for $X=[n]$, $R^{\otimes [n]}$ is nothing but $R^{\otimes (n+1)}$ (the last map in (\ref{equ:tensor}), sends a sequence $(X_1, \ldots, X_n)$ to $X_1\otimes\cdots \otimes X_n$).
\end{construction}
Recall that for a symmetric monoidal $\infty$-category $\calC$,  we have an equivalence 
\begin{equation*}
    \CAlg(\CAlg(\calC)) \simeq \CAlg(\calC).
\end{equation*}
\begin{definition}\label{def:Loday}
Let $\calC$ be presentable stable symmetric monoidal $\infty$-category. Given a commutative algebra object in $\calC$, $R\in \CAlg(\CAlg(\calC)) \simeq \CAlg(\calC)$.
For any finite simplicial set $X$, the Loday construction is defined to be \begin{equation*}
     R^{\otimes X}:= \colim( \triangle^{op} \overset{X}{\to}  \mathrm{Fin} \overset{R}{\to} \CAlg(\calC)^{\otimes}_{act} \overset{\otimes}{\to} \CAlg(\calC)).
\end{equation*}
Note that any simplicial set $Y$ can be written as a colimit of finite simplicial sets. Hence for $Y$, we write $Y\simeq \colim Y_i$. Then form the colimit 
\begin{equation*}
    R^{\otimes Y}:= \colim R^{\otimes Y_i}.
\end{equation*}
\end{definition}
A typical example is $\calC=\Sp$ and $X=S^1$. This gives the topological Hochschild homology.
\begin{theorem}[{\cite{MSV97}}]
Given a commutative ring spectrum $A$. We have an equivalence of commutative algebras
\begin{equation*}
    S^1\otimes A\simeq \THH(A).
\end{equation*}
\end{theorem}
We can also consider other finite simplicial sets rather than $S^1$, for instance $T^n=(S^1)^{\times n}$. Hence we define $L_{T^n}A=T^n\otimes A$. We can prove the following theorem.
\begin{theorem}
Given a commutative ring spectrum $A$. We have an equivalence of commutative algebras
\begin{equation*}
L_{T^n} A\simeq \THH^{(n)} (A), n\geq 1.
\end{equation*}
Here, $\THH^{(n)}(A):= \THH(\THH^{(n-1)}(A))$, $n\geq 2$. $\THH^{(1)}(A)=\THH(A)$.
\end{theorem}
\begin{proof}
We do the induction on $n$. When $n=1$, the equivalence is obviously hold. 
Suppose $n=k$ we have a such equivalence, we need to produce an equivalence for $n=k+1$.
Note that we have following equivalences
$L_{T^{k+1}} A= T^{k+1}\otimes A \simeq (S^1)^{\times (k+1)} \otimes A\simeq (S^1\times T^k)\otimes A$. 
By the definition we get
\begin{equation*}
  (S^1\times T^k)\otimes A\simeq S^1\otimes(T^k\otimes A).
\end{equation*}

Therefore, we have $L_{T^{k+1}} A\simeq \THH(L_{T^k} A)\simeq \THH^{(k+1)} A$. 
\end{proof}
So we can see that there is a $T^n$-action on $L_{T^n} A$ (or $\THH^{(n)}(A)$).

Note that, $T^n\otimes A\simeq \colim_k (\suspen_{+} T^n_k\otimes A)\simeq (\suspen_{+} \colim_{k} T^n_k)\otimes A\simeq \suspen_+ T^n\otimes A $. The later term is obviously a module over $\suspen_{+} T^n$ (in fact an algebra over $\suspen_+ T^n$), i.e. $T^n\otimes A\in \Mod_{\suspen_{+} T^n}\simeq \Fun(BT^n,\Sp)$.

The last equivalence $\Mod_{\suspen_{+} T^n} \simeq \Fun(BT^n, \Sp)$ can be found in \cite[Page 2]{Ath}.

It says that given a connected space $X$, the category $\Sp^X:= \Fun(X, \Sp)$ is equivalent to  $\Mod_{\mathrm{End}(i_! \Ss)}$. The spectrum $\mathrm{End}(i_!\Ss)$ can be identified as $\suspen_+ \Omega X $. Setting $X= BT^n$, this yields \begin{equation*}
    \Sp^{B T^n} \simeq \Mod_{\suspen_{+} \Omega B T^n} \simeq \Mod_{\suspen_{+} T^n}.
\end{equation*}

So $L_{T^n}A$ can be called higher topological Hochschild homology. We can also define the relative higher topological Hochschild homology.

In fact, this $T^n$-action can be produced more explicitly.

\begin{construction}
Let $\calC$ be a presentable stable symmetric monoidal $\infty$-category. Given a commutative algebra $R\in \CAlg(\calC)$.
We define $R^{\otimes T^n}$ or $L_{T^n}R$ to be the colimit of the following diagram
\begin{equation}\label{equ:highertori}
 (\Lambda^{op})^{\times n} \overset{V^o\times \cdots V^o}{\to} \mathrm{Fin}^{\times n}  \to \mathrm{Fin} \overset{R}{\to} \CAlg(\calC)_{act}^{\otimes} \overset{\otimes}{\to} \CAlg(\calC).
\end{equation}
Here $V^o: \Lambda^{op} \simeq \Lambda \overset{V}{\to} \mathrm{Fin}$ (c.f. \cite[Page 145]{NS18}).
Since there is a natural functor \cite[Proposition B.5]{NS18} \begin{equation*}
    \Fun(N(\Lambda^{op}),\calC) \to \Fun(B S^1, \calC)
\end{equation*}
given by taking the geometric realization.
\end{construction}
Note that there is an equivalence $\Fun(A\times B,\calC) \simeq \Fun(A,\Fun(B,\calC))$ for any $\infty$-categories $A,B, \calC$.
Thus, the colimit of (\ref{equ:highertori}) yields an object in $\Fun(B T^n,\CAlg(\calC))$.
\begin{definition}\label{def:relativetorihomology}
Let $B$ be a commutative ring spectrum. For any commutative $B$-algebra $A$, i.e. $A\in\CAlg(\Mod_B)$. We define $L_{T^n}^B A:= T^n\otimes_{B} A$.
For same reason, we can prove $L_{T^n}^B A\simeq \THH^{(n)}(A/B)$. Here, $\THH^{(n)}(A/B):= \THH(\THH^{(n-1)}(A/B)/B)$, $\THH^{(1)}(A/B)=\THH(A/B)$.
\end{definition}

Suppose $k$ is a perfect field of characteristic $p$. Then we have the following theorem.
\begin{definition}
Let $B_1=\Ff_p[\mu], |\mu=2|$. For $n\geq 2$, set 
\begin{equation*}
    B_n= \mathrm{Tor}^{B_{n-1}}_*(\Ff_p,\Ff_p).
\end{equation*}
\end{definition}
\begin{theorem}
When $n\leq 2p$, there is an $k$-Hopf algebra isomorphism
\begin{equation*}
    \pi_*(L_{S^n}k)\simeq B_n\otimes_{\Ff_p}k.
\end{equation*}
\end{theorem}
\begin{proof}
    This follows from the following lemma and \cite[Proposition 3.6]{Veen18}.
\end{proof}
\begin{lemma}
There is an equivalence
\begin{equation*}
    L_{S^n} k\simeq L_{S^n}\Ff_p\otimes_{\Ff_p} k
\end{equation*}
\end{lemma}
\begin{proof}
We prove this by induction on $n$.
\begin{itemize}
    \item $n=1$, this is true by \cite[Corollary 5.5]{HM97}.
    \item Note that $S^n$ can be written as $D^n\cup_{S^{n-1}} D^n$, hence by \cite[Proposition 2.2]{Veen18}  we have equivalences \begin{equation*} 
        L_{T^n}k \simeq k\otimes_{L_{T^{n-1}}k} k\simeq 
        (\Ff_p\otimes_{\Ff_p}k)\otimes_{L_{T^{n-1}}\Ff_p \otimes_{\Ff_p}k }(\Ff_p\otimes_{\Ff_p}k)
        \simeq L_{T^n}\Ff_p\otimes_{\Ff_p} k.
    \end{equation*}
\end{itemize}
\end{proof}

It seems that we should we the following equivalence,
\begin{equation*}
    \pi_*L_{T^n} k\simeq \pi_* L_{T^n}\Ff_p\otimes_{\Ff_p}k.
\end{equation*}

\begin{question}\label{question:THHn}
Given $1\leq n\leq p$ when $p\geq 5$ and $1\leq n\leq 2$ when $p=3$.
    Do we have an equivalence of graded $\Ff_p$-algebras ?
    \begin{equation*}
    \pi_*L_{T^n} k\simeq \bigotimes_{U\subseteq \mathbf{n}} B_U
    \end{equation*}
\end{question}
\begin{remark}
Suppose $k$ is a perfect field of characteristic $p$, then this is true for $n=2$. 
We have the following formula
\begin{align*}
 \pi_*(L_{T^2}k) & \simeq \THH_*(\THH(k))\simeq
    \THH_*(k[\Omega S^3]) \\
    & \simeq \pi_*(\THH(k)\otimes_{k} \HH(k[t]/k))\\
    & \simeq k[u]\otimes k[t] \otimes \Lambda_{k}(\sigma t).
\end{align*}

The second equivalence is obtained by \cite{THHThom} and \cite[Corollary 5.5]{HM97}. 
Here $|t|=2, |\sigma t|=3$. 
\end{remark}

\section{Homotopy fixed points of higher topological Hochschild homology}\label{section3}

Veen proved the following theorem.
\begin{theorem}[{\cite{Veen18}}]
Let $p\geq 5$ and $1\leq n\leq p$ or $p=3$ and $1\leq n\leq 2$.
The unit map \begin{equation*}
    k(n-1)_*\to k(n-1)_*(L_{T^n} \Ff_p)^{h T^n}
\end{equation*}
sends $v_{n-1}$ to a nonzero element in the right hand side.
\end{theorem}

In particular, under the conditions of above theorem, $k(n-1)_* (L_{T^n} \Ff_p)^{h T^n}\ne 0$. Note that, there is a ring map
\begin{equation*}
    k(n-1)\otimes(L_{T^n}\otimes \Ff_p)^{h T^n} \to K(n-1)\otimes (L_{T^n}\otimes \Ff_p)^{h T^n}.
\end{equation*}
So, if the right hand side is nonzero, then the left hand side is nonzero. In the contrast, we can not determine the right hand side if we only know the information about left hand side.
We can ask the following question.
\begin{problem}
What is the height of the commutative ring spectrum $(L_{T^n}\Ff_p)^{h T^n}$?
\end{problem}
Combining the known results and trace method, we can prove the following theorem.
\begin{theorem}
The height of commutative algebra $(L_{T^n} \Ff_p)^{h T^n}$ is at most $n-1$.
\end{theorem}
\begin{proof}
    Note that we have the following map of commutative algebras (because the trace map is lax symmetric monoidal).
    \begin{equation*}
    \K^{(n)}(\Ff_p) \to (L_{T^n} \Ff_p)^{h T^n}
    \end{equation*}
    Apply $K(n)\otimes -$ to the above map, we get a ring map
    \begin{equation*}
        K(n)\otimes \K^{(n)}(\Ff_p) \to K(n)\otimes (L_{T^n}\Ff_p)^{h T^n}.
    \end{equation*}
The left hand side is zero (c.f. \cite{vanishing}, \cite{purity}), thus the right hand side is zero.
\end{proof}
In general, we can ask the following question.
\begin{problem}\label{pro:height2n}
Given a commutative ring spectrum $R$ of height $n$. What is the height of commutative ring spectrum $(L_{T^n} R)^{h T^n}$? Is the height of $(L_{T^n} R)^{h T^n}$ exactly $2n$? 
\end{problem}
Similarly, according to the chromatic redshift theorem, we obtain the following theorem.
\begin{theorem}
Given a connective commutative ring spectrum $R$ of height $n$. The height of $(L_{T^n} R)^{h T^n}$ is at most $2n$.
\end{theorem}
\begin{proof}
Only need to noticing that the height of $K(2n+1)_*\K^{(n)}(R)\simeq 0$ ((c.f. \cite{vanishing}, \cite{purity})). Therefore $K(2n+1)_*(L_{T^n} R)^{h T^n}\simeq 0$.    
\end{proof}

We first see examples of height $0$ spectra.

\begin{theorem}
Let $p\geq 5$ and $1\leq n+1\leq p$, or $p=3$ and $1\leq n+1\leq 2$. Then the unit maps
\begin{align*}
& k(n)_* \to k(n)_* (L_{T^n}\Zz)^{h T^n}\\
& k(n)_* \to k(n)_*(L_{T^n} \BP\langle 0\rangle)^{h T^n}    
\end{align*}
send $v_n$ to a nonzero element for both cases.
\end{theorem}
\begin{proof}
By Quillen's computation \cite{Quillen72}, we know that $\K(\Ff_p)^{\wedge}_p\simeq H\Zz_p$. We then have following maps of commutative algebras
\begin{align*}
& H\Zz_p\simeq \K(\Ff_p)^{\wedge}_p \to (\THH(\Ff_p)^{hS^1})^{\wedge}_p\simeq \THH(\Ff_p)^{h S^1}\\
& H\Zz\to \BP\langle 0\rangle \to H\Zz_p.
\end{align*}
Therefore , we get maps of commutative algebras $H\Zz\to \BP\langle 0\rangle\to \THH(\Ff_p)^{h S^1}$. These yield ring maps
\begin{equation*}
 (L_{T^{n-1}}\Zz)^{h T^{n-1}}\to (L_{T^{n-1}}\BP\langle 0\rangle)^{h T^{n-1}} \to (L_{T^n} \Ff_p)^{h T^n}.
\end{equation*}
Apple $k(n-1)_*$. The right hand side is nonzero, therefore the two rings on the left hand side are nonzero.
\end{proof}

\begin{remark}
    Alternatively, we have map of commutative algebras \cite[Corollary IV.4.10]{NS18}
    \begin{equation*}
        H\Zz_p\simeq \tau_{\geq 0} \TC(\Ff_p)\to \THH(\Ff_p)^{h S^1}.
    \end{equation*}
Because the truncation functor $\tau_{\geq 0}$ is a lax symmetric monoidal functor right adjoint to the fully faithful inclusion functor. 
\end{remark}
\begin{remark}
Suppose Question \ref{question:THHn} is true. Then given a perfect field $k$ of characteristic $p$, we may prove that 
\begin{equation*}
    k(n-1)_* (L_{T^n} k)^{h T^n} \ne 0.
\end{equation*}
\end{remark}

\begin{theorem}
Suppose Question \ref{question:THHn} is true.
Let $p\geq 5$ and $1\leq n\leq p$, or $p=3$ and $1\leq n\leq 2$.  Given a perfect field $k$ of characteristic $p$, then the unit map
\begin{align*}
& k(n-1)_* \to k(n-1)_* (L_{T^n} k)^{h T^n}\\  
\end{align*}
send $v_{n-1}$ to a nonzero element.
\end{theorem}
\begin{proof}
There is a ring map
\begin{equation*}
   (L_{T^n} \Ff_p)^{h T^n} \to (L_{T^n}k)^{h T^n}.
\end{equation*}
The spectral sequence computing $F(E_2T^n_+, L_{T^n} \Ff_p)^{T^n}$ maps an infinite cycle to an infinite cycle. Now, the proposition follows from the following lemma.
\end{proof}

\begin{lemma}Suppose Question \ref{question:THHn} is true.
Let $\mu_i\in \pi_2 L_{T^n}k$ be the image of the generator 
in $\pi_2L_{S^1_i}k \simeq k\{\mu_i\}$ under the map \begin{equation*}
    L_{S_i^1} k \to L_{T^n} k.
\end{equation*}
Then the Rognes class $\Sigma_{i=1}^n t_i\mu_i^{p^{n-1}}$ in the $k(n)$-homology fixed point spectral sequence is not hit by any differential.
\begin{equation*}
   \Zz\{t_1, t_2, \ldots, t_n\} \otimes k(n-1)_* (L_{T^n} k)
   \Longrightarrow k(n-1)_* F(E_2T^n_+, L_{T^n} k)^{T^n}.
\end{equation*}
\end{lemma}
\begin{proof}
Only need to note that $H_*L_{T^n} k\simeq H_*\Ff_p\otimes_{\Ff_p} \pi_*L_{T^n} \Ff_p\otimes_{\Ff_p}k \simeq k\otimes_{\Ff_p}\calA\otimes_{\Ff_p} \pi_*L_{T^n}\Ff_p$.

Now run the same arguments in \cite[Proposition 7.4, 7.5 7.7]{Veen18}.
\end{proof}

Furthermore, considering that $\K(\mathbb{Z})$ and $\K^{(2)}(\mathbb{F}_p) = \K(\K(\mathbb{F}_p))$ are ring spectra of height 1, we can anticipate the following theorem.

\begin{theorem}
Let $p \geq 5$ and $1 \leq n+2 \leq p$. Then the unit maps
\begin{equation*}
k(n+1)_* \to k(n+1)_* (L_{T^n} \K(\Zz))^{h T^n}
\end{equation*}
and
\begin{equation*}
k(n+1)_* \to k(n+1)_* (L_{T^n} \K^{(2)}(\Ff_p))^{h T^n}   
\end{equation*}
map $v_n$ to a non-zero element.
\end{theorem}

\begin{proof}
Based on the previous discussion, there are maps of commutative ring spectra:
\begin{align*}
& \K(\Zz) \to \TC^-(\Zz), \\
& \Zz \to \TC^-(\Ff_p).
\end{align*}
The second map induces the following map between commutative ring spectra:
\begin{equation*}
\THH(\mathbb{Z})^{h S^1} \to (L_{T^2} \Ff_p)^{h T^2}.
\end{equation*}
Consequently, we have a map of commutative ring spectra:
\begin{equation*}
\K(\Zz) \to \TC^-(\Ff_p).
\end{equation*}
This yields a map of commutative ring spectra:
\begin{equation*}
(L_{T^n} \K(\Zz))^{h T^n} \to (L_{T^{n+2}} \Ff_p)^{h T^{n+2}}.
\end{equation*}
Applying $k(n+1)_*$ to both sides, since the right-hand side is non-zero, the left-hand side must also be non-zero. This proves the first part of the proposition.

Similarly, via the trace map, we obtain a map of commutative ring spectra:
\begin{equation*}
(L_{T^n} \K^{(2)}(\mathbb{F}_p))^{h T^n} \to (L_{T^{n+2}} \mathbb{F}_p)^{h T^{n+2}}.
\end{equation*}
Applying $k(n+1)_*$ to both sides, since the right-hand side is non-zero, the left-hand side is non-zero as well. This completes the proof of the other part of the proposition.
\end{proof}
We use the following lemma to investigate $k(n+1)\otimes (L_{T^n} j_p)^{h T^n}$.

\begin{lemma}[{c.f. \cite[Lemma 1.3.1]{devalapurkar2025thhzimagej}}]\label{lemma:jmap}
There is a map of commutative algebras
    \begin{equation*}
        j_p\to \K(\Zz_p)_p.
    \end{equation*}
\end{lemma}
\begin{proof}
The natural map $\K(\Zz_p)_p\to L_{K(1)}\K(\Zz_p)_p$ is $2$-connective (the fiber lies in $\Sp_{\leq 1}$).
We then have a pullback square of commutative algebras
\begin{equation*}
\xymatrix{
\tau_{\geq 0}K(\Zz_p)_p \ar[r]\ar[d] & \tau_{\geq 0} L_{K(1)}\K(\Zz_p)_p\ar[d]\\
\tau_{\leq 1}\tau_{\geq 0}K(\Zz_p)_p \ar[r] & \tau_{\leq 1} \tau_{\geq 0}L_{K(1)}\K(\Zz_p)_p
}
\end{equation*}
Hence it suffice to produces maps of commutative algebras from $j_p$ to other three corners.
\begin{itemize}
    \item There is an unique map of commutative algebras from $j_p\simeq \tau_{\geq 0} L_{K(1)}\Ss\to \tau_{\geq 0} L_{K(1)} \K(\Zz_p)_p$. Because $L_{K(1)}\Ss$ is the unit.
    \item For any $p$-complete commutative algebra $A$ lies in $\Sp_{[0,1]}$, there is a unique map from $j_p$ to $A$. Because $\tau_{\leq 1}\Ss_p\simeq \tau_{\leq 1} j_p$.
\end{itemize}
\end{proof}
\begin{theorem}\label{thm:higherredshift}
Let $p\geq 5$ and $1\leq n+2\leq p$. Then the unit map
\begin{equation*}
k(n+1)_* \to k(n+1)_* (L_{T^n} j_p)^{h T^n}
\end{equation*}
sends $v_n$ to a nonzero element.
\end{theorem}
\begin{proof}
By Lemma \ref{lemma:jmap}, 
we have maps of commutative algebras
\begin{align*}
& j_p\to K(\Zz_p)_p \to (\THH(\Zz_p)^{h S^1})^{\wedge}_p.\\
& \THH(\Zz_p/\Ss_{\Zz_p})^{h S^1} \to (L_{T^2}^{\Ss_{\Zz_p}} \Ff_p)^{h T^2}.
\end{align*}
The second row map comes from the following maps
\begin{equation*}
    \THH(\Zz_p/\Ss_{\Zz_p})^{h S^1} \to \THH(\THH(\Ff_p/\Ss_{\Zz_p})^{h S^1}/\Ss_{\Zz_p})^{h S^1} \to \THH(\THH(\Ff_p/\Ss_{\Zz_p})/\Ss_{\Zz_p})^{h T^2}\simeq (L_{T^2}^{\Ss_{\Zz_p}}\Ff_p)^{h T^2}.
\end{equation*}
Note that $\TC^-(\Zz_p)^{\wedge}_p\simeq \TC^-(\Zz_p/\Ss_{\Zz_p})$, see \cite[Remark 2.13]{LW22}, \cite[Remark 2.2]{KN22}.
Hence we have a commutative ring map 
\begin{equation*}
  j_p\to ((L_{T^2}^{\Ss_{\Zz_p}} \Ff_p)^{h T^2}).  
\end{equation*}
Note that we have  equivalences
\begin{equation*}
 L_{T^2}^{\Ss_{\Zz_p}} \Ff_p\simeq \THH(\THH(\Ff_p/\Ss_{\Zz_p})/\Ss_{\Zz_p})\simeq \THH(\THH(\Ff_p)/\Ss_{\Zz_p})\simeq \THH(\THH(\Ff_p))\simeq L_{T^2}\Ff_p.
\end{equation*}
The first equivalence follows from Definition \ref{def:relativetorihomology}. The second equivalence follows from the following equivalences 
\begin{equation*}
   \Ff_p\otimes_{\Ss_{\Zz_p}}\Ff_p\simeq \Ff_p\otimes_{\Ss_{\Zz_p}}(\Ss_{\Zz_p} \otimes\Ff_p)\simeq \Ff_p\otimes \Ff_p 
\end{equation*}

The third equivalence follows from the following equivalences 
\begin{equation*}
\THH(\Ff_p)\otimes_{\Ss_{\Zz_p}}\THH(\Ff_p)\simeq \THH(\Ff_p)\otimes_{\Ss_{\Zz_p}}(\Ss_{\Zz_p} \otimes(\Ff_p\otimes_{\Ff_p}\THH(\Ff_p)))\simeq \THH(\Ff_p)\otimes \THH(\Ff_p).    
\end{equation*}

Therefore, we finally have a commutative ring map
\begin{equation*}
    j_p\to (L_{T^2} \Ff_p)^{h T^2}.
\end{equation*}

Hence, we have a ring map
\begin{equation*}
(L_{T^n} j_p)^{h T^n} \to (L_{T^{n+2}}\Ff_p)^{h T^{n+2}}.
\end{equation*}
Applying $k(n+1)_*$ to the above map,  since the right hand side is nonzero, so the left hand side is nonzero.
\end{proof}
\begin{remark}
    We can produce a ring map
    \begin{equation*}
        j_p\to ((L_{T^2} \Ff_p)^{h T^2})^{\wedge}_p.
    \end{equation*}
However, it seems that we can not show that $(L_{T^2}\Ff_p)^{h T^2}$ is $p$-complete (essentially because $\TC^-(\Ff_p)$ is not bounded below).
\end{remark}
\begin{remark}\label{trick}
For a group $G$, let $\Sp^{BG}=\Fun(BG,\Sp)$. Then for a commutative algebra object $A\in \CAlg(\Sp^{BH})$. We have an adjunction pair 
\begin{equation*}
    U: \Sp^{BH} \rightleftarrows \Sp: (-)^{h H}.
\end{equation*}
Where, the forgetful functor $U$ is symmetric monoidal.
We have the Loday construction $G\otimes - :\CAlg(\Sp^{BH}) \to \CAlg(\Sp^{B (H\times G)})$. Let $X\in \CAlg(\Sp^{B H})$.
Therefore we have a natural transformation
\begin{equation*}
    (G\otimes X^{h H}) \to (G\otimes X)^{h H}.
\end{equation*}
Applying $-^{h G}$, we get a natural transformation
\begin{equation*}
    (G\otimes X^{h H})^{h G} \to (G\otimes X)^{h H\times G}.
\end{equation*}
\end{remark} 

The above statements also suggests the following proposition.

\begin{proposition}
Let $R$ be a commutative ring spectrum. If there is a commutative ring map
\begin{equation*}
    R\to \THH(\Zz_p)^{h S^1}.
\end{equation*}
Let $p\geq 5$ and $1\leq n+2\leq p$. Then the unit map
\begin{equation*}
k(n+1)_* \to k(n+1)_* (L_{T^n} R)^{h T^n}
\end{equation*}
sends $v_n$ to a nonzero element.
\end{proposition}

\begin{proof}
    From the condition on $R$, we can obtain a commutative ring map
    \begin{equation*}
        R\to (L_{T^2} \Ff_p)^{h T^2}.
    \end{equation*}
This gives a commutative ring map
\begin{equation*}
   (L_{T^n} R)^{h T^n} \to (L_{T^{n+2}}\Ff_p)^{h T^{n+2}}.
\end{equation*}
Applying $k(n+1)_*$, since the right hand side is non-zero, hence the left hand side is non-zero.
\end{proof}
\section{Further direction}\label{section4}
It seems that the method presented in this paper cannot applied to $\BP\langle 1\rangle$ or $ku_p$.
In other words, we can not produce a ring map
\begin{equation*}
    ku_p\to \THH(\Zz_p)^{h S^1}.
\end{equation*}
However, we have $\Ee_{\infty}$-ring maps \cite[Construction 3.7]{kulocal-k}
\begin{equation*}
    L_{K(1)}ku_p \to 
    L_{K(1)} \K(\Zz_p[\zeta_{p^{\infty}}])
    \to  L_{K(1)}\TC (\Zz_p[\zeta_{p^{\infty}}])
\end{equation*}
So, the following questions are interesting.
\begin{question}
\begin{enumerate}
    \item Is \begin{equation*}
    k(n)_*(L_{T^n} \Zz_p[\zeta_{p^\infty}])^{h T^n}
\end{equation*}
non-zero.
\item Is \begin{equation*}
    K(n)_*(L_{T^n} \Zz_p[\zeta_{p^\infty}])^{h T^n}
\end{equation*}
non-zero.
\end{enumerate}
\end{question}
According to the chromatic Nullstellensatz, we only need to investigate the cases for Lubin--Tate spectra.
So, one can ask the following questions.
Let $e_n:=\tau_{\geq 0} E_n$. 
\begin{question}
\begin{enumerate}
    \item Is there a height $n-1$ commutative ring spectrum $R_{n-1}$, such that there is a commutative ring map
\begin{equation*}
    e_n \to \THH(R_{n-1})^{h S^1}.
\end{equation*}
\item Is there a height $n-1$ commutative ring spectrum $R_{n-1}$, such that there is a commutative ring map
\begin{equation*}
    e_n \to \TC(R_{n-1})^{\wedge}_p.
\end{equation*}
\end{enumerate}
\end{question}

\bibliographystyle{alpha}
\bibliography{ref}
\end{document}